\documentclass[10pt]{article}

\usepackage{amsmath}
\usepackage{amsfonts}
\usepackage{amsthm}
\usepackage[english]{babel}

\newtheorem{theorem}{Theorem}
\newtheorem{lemma}{Lemma}

\newtheorem{corollary}{Corollary}

\newtheorem{remark}{Remark}

\newcommand{\mR}{\mathbb{R}}
\newcommand{\mC}{\mathbb{C}}
\newcommand{\mN}{\mathbb{N}}
\newcommand{\mE}{\mathbb{E}}

\newcommand{\hD}{\mathcal{D}}
\newcommand{\cM}{\mathcal{M}}

\newcommand{\cF}{\mathcal{F}}
\newcommand{\cP}{\mathcal{P}}
\newcommand{\cC}{\mathcal{C}}

\newcommand{\ux}{\underline{x}}
\newcommand{\uxb}{\underline{x} \grave{}}

\newcommand{\uyb}{\underline{y} \grave{}}
\newcommand{\uy}{\underline{y}}

\newcommand{\ds}{d \sigma_{\ux}}
\newcommand{\dsb}{d \sigma_{\uxb}}
\newcommand{\dsx}{d \sigma_{x}}

\newcommand{\pj}{\partial_{x_j}}
\newcommand{\pjb}{\partial_{{x \grave{}}_{j}}}

\newcommand{\px}{\partial_x}

\newcommand{\upx}{\partial_{\underline{x}}}
\newcommand{\upxb}{\partial_{\underline{{x \grave{}}} }}

\begin{document}

\title{A Cauchy integral formula in superspace} 

\author{H.\ De Bie\footnote{Corresponding author, E-mail: {\tt Hendrik.DeBie@UGent.be}} \and F.\ Sommen\footnote{E-mail: {\tt fs@cage.ugent.be}}
}

\date{\small{Clifford Research Group -- Department of Mathematical Analysis}\\
\small{Faculty of Engineering -- Ghent University\\ Krijgslaan 281, 9000 Gent,
Belgium}}

\maketitle

\begin{abstract}
In previous work the framework for a hypercomplex function theory in superspace was established and amply investigated. In this paper a Cauchy integral formula is obtained in this new framework by exploiting techniques from orthogonal  Clifford analysis. After introducing Clifford algebra valued surface- and volume-elements first a purely fermionic Cauchy formula is proven. Combining this formula with the already well-known bosonic Cauchy formula yields the general case. Here the integration over the boundary of a supermanifold is an integration over as well the even as the odd boundary (in a formal way). Finally, some additional results such as a Cauchy-Pompeiu formula and a representation formula for monogenic functions are proven.
\end{abstract}


\section{Introduction}
Superspaces and more general supermanifolds play an important role in contemporary theoretical physics, e.g. in the study of supersymmetric (gauge), supergravity or superstring theories, in the study of the geometrical meaning of the BRST symmetry, as well as in the theory of random matrices etc.

As to the mathematical part of the theory, several approaches are possible. Superspaces were first introduced by Berezin, see \cite{MR0208930,MR732126}. His approach was deeply influenced by modern algebraic geometry using schemes and sheaf theory. Other important references in this approach are \cite{MR565567,MR0580292}. Later, supermanifold theory was also studied from the point of view of differential geometry (which perhaps ties in better with a physical way of thinking). We refer the reader to \cite{MR778559,MR574696} and the book \cite{MR1175751} for a general overview. Moreover, both approaches are equivalent in the categorical sense, as was shown in \cite{MR554332}.

It is however possible to study superspaces from yet another point of view, namely that of harmonic analysis and hypercomplex function theory (as a refinement of harmonic analysis). In a set of recent papers we have extended the theory of Clifford analysis (see e.g. \cite{MR697564,MR1130821,MR1169463}) to superspace in a canonical way. In \cite{DBS1,DBS4} we have established the computational framework, introducing the basic symbols such as variables and Clifford numbers, and also the basic differential operators (Dirac and Laplace operators, Euler and Gamma operators, etc.).

Important in this new approach to superspace is the introduction of the so-called super-dimension $M$, which gives a global characterization of super Euclidean spaces; it is defined by the action of the Dirac operator on the vector variable. A lot of results in superspace may be found by simply replacing the classical dimension by this super-dimension in special formulae. This technique was used in e.g. \cite{DBS5}, where we studied spherical harmonics in superspace and introduced an integral over the supersphere (based on an old result of Pizzetti). Moreover, this integral was extended to the whole superspace by a generalized form of integration in spherical co-ordinates and it was proven that this yields the same result as the well-known Berezin integral (see e.g. \cite{MR0208930,MR732126}). 

One of the most interesting features of Clifford analysis is that it allows for the construction of several nice Cauchy-type formulae in higher dimensions (see e.g. \cite{JR3,MR1012510,JR2}). The aim of the present paper is hence to show that one can also obtain a generalization to superspace of Cauchy's integral formula.
This is an important result, because it is more or less equivalent to asking for a formula of Stokes in superspace, connecting integration over a supermanifold with integration over its boundary. 

Note that there are versions of Stokes' formula known for supermanifolds, see e.g. the
work of Palamodov in \cite{MR1402921}. In that approach however, rather complicated machinery
of algebraic geometry is used. Our aim is to use instead the framework of hypercomplex analysis to obtain a Cauchy formula in a more straightforward way. The advantage of this Clifford analysis framework is that it will allow us to predict the form of the desired formula by analogy with the classical case. Moreover we will obtain that the boundary of a supermanifold consists of two parts, which can be interpreted as the even and the odd boundary.

As a consequence of this integral formula, we will be able to construct a Cauchy-Pompeiu formula and a Cauchy representation formula for monogenic functions in superspace (i.e. null-solutions of the super Dirac operator), although not all nice properties from the complex plane will be preserved (such as e.g. Morera's theorem).

The paper is organized as follows. We start with introducing the basic operators and function spaces in section \ref{preliminaries}. In section \ref{limitcases} we first briefly discuss the Clifford analysis version of Stokes'  formula in $\mR^m$. This will provide us with the necessary ideas to construct a similar formula in purely fermionic space (i.e. the case where only anti-commuting co-ordinates are considered). In section \ref{generalthm} the general case is considered, which necessitates the construction of an appropriate surface-element and a corresponding volume-element. Finally, in section \ref{conseq} a few corollaries to this result, such as a Cauchy theorem, are discussed.

\section{Preliminaries}
\label{preliminaries}

The basic algebra of interest in the study of Clifford analysis in superspace (see \cite{DBS1,DBS4}) is the real algebra $\cP = \mbox{Alg}(x_i, e_i; {x \grave{}}_j,{e \grave{}}_j)$, $i=1,\ldots,m$, $j=1,\ldots,2n$
generated by

\begin{itemize}
\item $m$ commuting variables $x_i$ and $m$ orthogonal Clifford generators $e_i$
\item $2n$ anti-commuting variables ${x \grave{}}_i$ and $2n$ symplectic Clifford generators ${e \grave{}}_i$
\end{itemize}
subject to the multiplication relations
\[ \left \{
\begin{array}{l} 
x_i x_j =  x_j x_i\\
{x \grave{}}_i {x \grave{}}_j =  - {x \grave{}}_j {x \grave{}}_i\\
x_i {x \grave{}}_j =  {x \grave{}}_j x_i\\
\end{array} \right .
\quad \mbox{and} \quad
\left \{ \begin{array}{l}
e_j e_k + e_k e_j = -2 \delta_{jk}\\
{e \grave{}}_{2j} {e \grave{}}_{2k} -{e \grave{}}_{2k} {e \grave{}}_{2j}=0\\
{e \grave{}}_{2j-1} {e \grave{}}_{2k-1} -{e \grave{}}_{2k-1} {e \grave{}}_{2j-1}=0\\
{e \grave{}}_{2j-1} {e \grave{}}_{2k} -{e \grave{}}_{2k} {e \grave{}}_{2j-1}=\delta_{jk}\\
e_j {e \grave{}}_{k} +{e \grave{}}_{k} e_j = 0\\
\end{array} \right .
\]
and where moreover all elements $e_i$, ${e \grave{}}_j$ commute with all elements $x_i$, ${x \grave{}}_j$.

\noindent
If we denote by $\Lambda_{2n}$ the Grassmann algebra generated by the anti-commuting variables ${x \grave{}}_j$ and by $\cC$ the algebra generated by all the Clifford numbers $e_i, {e \grave{}}_j$, then we clearly have that
\[
\cP = \mR[x_1,\ldots,x_m]\otimes \Lambda_{2n} \otimes \cC.
\]

In the case where $n = 0$ we have that $\cC \cong \cC l_{0,m}$, the standard orthogonal Clifford algebra with signature $(-1,\ldots,-1)$. Similarly, the algebra generated by the ${e \grave{}}_j$ is isomorphic with the Weyl algebra over a vectorspace of dimension $2n$ equipped with the canonical symplectic form.

The most important element of the algebra $\cP$ is the vector variable $x = \ux+\uxb$ with
\[
\begin{array}{lll}
\ux &=& \sum_{i=1}^m x_i e_i\\
&& \vspace{-2mm}\\
\uxb &=& \sum_{j=1}^{2n} {x \grave{}}_{j} {e \grave{}}_{j}.
\end{array}
\]

One easily calculates that
\[
x^2 = \uxb^2 +\ux^2 = \sum_{j=1}^n {x\grave{}}_{2j-1} {x\grave{}}_{2j}  -  \sum_{j=1}^m x_j^2.
\]
The super Dirac operator is defined by
\begin{equation}
\px = \upxb-\upx = 2 \sum_{j=1}^{n} \left( {e \grave{}}_{2j} \partial_{{x\grave{}}_{2j-1}} - {e \grave{}}_{2j-1} \partial_{{x\grave{}}_{2j}}  \right)-\sum_{j=1}^m e_j \pj.
\label{leftDirac}
\end{equation}
If we let it act from the right, we have to introduce an extra minus sign (see \cite{DBS1})
\begin{equation}
\cdot \px = - \cdot \upxb - \cdot \upx.
\label{rightDirac}
\end{equation}

The square of the Dirac operator is the super Laplace operator:
\[
\Delta = \px^2 =4 \sum_{j=1}^n \partial_{{x \grave{}}_{2j-1}} \partial_{{x \grave{}}_{2j}} -\sum_{j=1}^{m} \pj^2.
\]
The bosonic part of this operator is $\Delta_b = -\sum_{j=1}^{m} \pj^2$, which is the classical Laplace operator. The fermionic part is $\Delta_f = 4 \sum_{j=1}^n \partial_{{x \grave{}}_{2j-1}} \partial_{{x \grave{}}_{2j}}$.

The Euler operator in superspace is defined as
\[
\mE = \sum_{j=1}^m x_j \pj+\sum_{j=1}^{2n} {x \grave{}}_{j} \pjb
\]
and allows us to decompose $\cP$ into spaces of homogeneous $\cC$-valued polynomials
\[
\cP = \bigoplus_{k=0}^{\infty} \cP_k, \quad \cP_k=\left\{ \omega \in \cP \; | \; \mE \omega=k \omega \right\}.
\]

For the other important operators in super Clifford analysis we refer the reader to \cite{DBS1,DBS4}. If we let $\px$ act on $x$ we find that
\[
\px x = x \px = m-2n = M
\]
where $M$ is the so-called super-dimension. This super-dimension is of the utmost importance (see e.g. \cite{DBS5}), as it gives a global characterization of our superspace. The physical meaning of this parameter is discussed in \cite{DBS3}.

The basic calculational rules for the Dirac operator on the algebra $\cP$ are given in the following lemma (see \cite{DBS1}).

\begin{lemma}
Let $s \in \mN$ and $R_k \in \cP_k$, then
\begin{eqnarray*}
\px(x^{2s} R_k) &=& 2 s x^{2s-1}R_k + x^{2s} \px R_k\\
\px(x^{2s+1} R_k) &=& (2k + M + 2s) x^{2s}R_k - x^{2s+1} \px R_k.
\end{eqnarray*}
\end{lemma}

If we define the space of spherical monogenics of degree $k$ by
\[
\cM_k =\left\{ R \in \cP_k \; | \; \px R = 0 \right\}
\]
we immediately have
\begin{corollary}
Let $s \in \mN$ and $P_k \in \cM_k$, then
\begin{eqnarray*}
\px(x^{2s} P_k) &=& 2 s x^{2s-1}P_k\\
\px(x^{2s+1} P_k) &=& (2k + M + 2s) x^{2s}P_k.
\end{eqnarray*}
\label{basicrel}
\end{corollary}

We also have the following lemma (see \cite{DBS5}).

\begin{lemma}
If $R_{2t} \in \cP_{2t}$, then the following holds:
\[
\Delta^{t+1}(x^2 R_{2t}) = 4(t+1)(M/2+t) \Delta^{t}( R_{2t}).
\]
\end{lemma}

From this formula we derive the following important result. Let $R_{2n-2k} \in \cP_{2n-2k}$ where the purely fermionic case ($m=0$, $M=-2n$) is considered. Then 
\begin{eqnarray*}
\upxb^{2n}(\uxb^{2k} R_{2n-2k}) &=& 4n (-n+n-1) \upxb^{2n-2}(\uxb^{2k-2} R_{2n-2k})\\
&=& 4n (-1) 4(n-1)(-n+n-2) \upxb^{2n-4}(\uxb^{2k-4} R_{2n-2k})\\
&=& 4^2 n(n-1) (-1)(-2) \upxb^{2n-4}(\uxb^{2k-4} R_{2n-2k})\\
&=& \ldots\\
&=&4^k n (n-1) \ldots (n-k+1) (-1) (-2) \ldots (-k) \upxb^{2n-2k}( R_{2n-2k})\\
&=& (-1)^k 4^k \frac{n! k! }{(n-k)!} \upxb^{2n-2k}( R_{2n-2k}).
\end{eqnarray*}
We define for further use the numerical coefficient
\[
c(n,k) = (-1)^k 4^k \frac{n! k! }{(n-k)!}.
\]

For the treatment of a Cauchy formula in superspace we need of course a broader set of functions. For our purposes we define the function spaces
\[
\cF(\Omega)_{m|2n} = \cF(\Omega) \otimes \Lambda_{2n} \otimes \cC
\]
where $\Omega$ is an open set in $\mR^m$ and where $\cF(\Omega)$ stands for $\hD(\Omega)$, $C^{k}(\Omega)$, $L_{p}(\Omega)$, $L_{1}^{\mbox{\footnotesize loc}}(\Omega)$, $\ldots$ according to the application. We denote by $\cM(\Omega)_{m|2n}^{l(r)} \subset 
C^{1}(\Omega)_{m|2n}$ the space of left (respectively right) monogenic functions, i.e. null-solutions of the super Dirac operator. 

The fundamental solution of the super Laplace operator is given by 
\begin{equation}
\nu_2^{m|2n} = \sum_{k=0}^n \frac{4^k k!}{(n-k)!} \nu_{2k+2}^{m|0} \uxb^{2n-2k},
\label{fundlapl}
\end{equation}
where the functions $\nu_{i}^{m|0}$ are fundamental solutions of the powers of the bosonic Laplace operator $\Delta_b$ (see e.g. \cite{MR745128} for some explicit expressions), satisfying
\begin{eqnarray*}
\Delta^j_b \nu_{2l}^{m|0} &=& \nu_{2l-2j}^{m|0}, \quad j<l\\
\Delta^l_b \nu_{2l}^{m|0} &=& \delta(\ux).
\end{eqnarray*}
This fundamental solution can be determined either by direct methods (see \cite{DBS6}) or by Fourier transform (see \cite{DBS9}).

Letting the super Dirac operator act on the left of (\ref{fundlapl}) (acting on the right yields the same result) gives us the left (and right) fundamental solution of the Dirac operator in superspace:
\begin{eqnarray}
\nu_{1}^{m|2n} &=& \sum_{k=0}^{n-1} 2 \frac{4^k k!}{(n-k-1)!} \nu_{2k+2}^{m|0} \uxb^{2n-2k-1} + \sum_{k=0}^n \frac{4^k k! }{(n-k)!} \nu_{2k+1}^{m|0} \uxb^{2n-2k}
\end{eqnarray}
with $\nu_{2k+1}^{m|0} = -\upx \nu_{2k+2}^{m|0} = -\nu_{2k+2}^{m|0} \upx$. Note that there does not exist a fundamental solution in the purely fermionic case (see \cite{DBS6}).  

Integration over superspace is given by the so-called Berezin integral (see \cite{MR0208930,MR732126}), defined by
\[
\int_{\mR^{m|2n}} f = \int_{\mR^m} \int_B f, \qquad f \in \hD(\Omega)_{m|2n}
\]
with
\[
\int_B = \partial_{{x \grave{}}_{2n}} \ldots \partial_{{x \grave{}}_{1}} = \frac{(-1)^n}{4^n n!} \upxb^{2n}.
\]
This definition means that one first has to derive $f$ with respect to all anti-commuting variables and then to integrate the commuting variables in the usual way. This integration recipe may seem rather haphazard but can be explained in a satisfactory way using harmonic analysis in superspace (see \cite{DBS5}).

\section{The limit cases}
\label{limitcases}

In this section we discuss the Clifford analysis versions of Stokes' theorem in the two limit cases, namely the purely bosonic case where one considers only commuting variables ($m \neq 0$, $n=0$) and the purely fermionic case where only anti-commuting variables are considered ($m = 0$, $n \neq 0$). We start with the bosonic case.

\subsection{The bosonic Stokes' theorem}

Let $\Omega$ be an open set in $\mR^m$, $\Sigma$ a compact oriented differentiable $m$-dimensional manifold in $\Omega$ and $\partial \Sigma$ its smooth boundary. Then by introducing the following vector-valued surface-element 
\[
\ds = \sum_{j=1}^m (-1)^{j+1} e_j \widehat{dx_j}, \quad \qquad \widehat{dx_j} = dx_1 \ldots dx_{j-1} dx_{j+1}\ldots dx_m
\]
and the volume-element
\[
dV(\ux) = dx_1 \ldots dx_m,
\]
where the exterior product of differential forms is understood, we have the following (classical) theorem (see e.g. \cite{MR697564}).

\begin{theorem}[(Bosonic Stokes' theorem)]
Let $f$ and $g$ be $C^1$-functions defined on $\Omega$ with values in $\cC l_{0,m}$. Then one has
\[
\int_{\partial \Sigma} f \ds \, g = \int_{\Sigma} [(f \upx)g + f (\upx g)] dV(\ux).
\]
\end{theorem}

The proof follows from a direct application of Stokes' theorem, because $d (f \ds \, g) =  [(f \upx)g + f (\upx g)] dV(\ux)$.

\subsection{The fermionic Stokes' theorem}

We want to construct a formula which looks like
\begin{equation}
\int_B f \dsb g =  \int_B  [( - (f \upxb)  g + f ( \upxb  g)  )] dV(\uxb)
\label{fermcauchyform}
\end{equation}
where $dV(\uxb)$ is a suitable volume-element in the purely fermionic superspace and $\dsb$ a corresponding surface-element. 
We first note that in the bosonic case, upon introducing $d \ux = \sum_{i=1}^{m} e_i dx_i$, the surface- and volume-elements can be expressed as follows:
\begin{equation}
\begin{array}{lll}
(d\ux)^m &=& m! \; e_1 \ldots e_m \; dV(\ux)\\
\vspace{-2mm}\\
(d\ux)^{m-1} &=& -(m-1)! \; \ds \; e_1 \ldots e_m.
\end{array}
\label{vectordiff}
\end{equation}

As we have shown in our paper \cite{DBS4}, by studying the homology of the super Hodge coderivative $d^*$ on spaces of polynomial valued differential forms, the correct volume-element in fermionic superspace is
\[
{x \grave{}}_{1} \ldots {x \grave{}}_{2n} = \frac{\uxb^{2n}}{n!}.
\]
Comparing this with formula (\ref{vectordiff}) we conclude that a good candidate for $\dsb$ would then be $\uxb^{2n-1}/ (n-1)!$.

Although this approach would indeed yield a Stokes' type formula in superspace, it is in fact a bit meager. Indeed, if we expand $f$ and $g$ into homogeneous components
\begin{eqnarray*}
f &=& f_0 + f_1 + \ldots f_{2n}, \qquad f_i \in \cP_i\\
g &=& g_0 + g_1 + \ldots g_{2n}, \qquad g_i \in \cP_i,
\end{eqnarray*}
we obtain
\[
f \uxb^{2n-1} g = f_0 \uxb^{2n-1} g_1 + f_1 \uxb^{2n-1} g_0,
\]
because there are no polynomials of degree higher than $2n$. In formula (\ref{fermcauchyform}), only the terms $f_0, f_1$ and $g_0, g_1$ of the functions $f$ and $g$ would thus play a role. This can be extended by instead introducing the following definitions:
\[
\dsb = -2 \left( \uxb + \frac{\uxb^3}{1!}+ \frac{\uxb^5}{2!} + \ldots + \frac{\uxb^{2n-1}}{(n-1)!} \right) = -2 \uxb \exp(\uxb^2)
\]
and
\[
dV(\uxb) =\frac{\uxb^2}{1!}+ \frac{\uxb^4}{2!} + \ldots + \frac{\uxb^{2n}}{n!} = \exp(\uxb^2)-1
\]
as will become clear in the sequel. In this way, more components of the functions $f$ and $g$ will contribute to the resulting formula (see theorem \ref{fermcauchythm}).

Let us start with the following technical lemma.
\begin{lemma}
Suppose $f \in \cP_i$, $g \in \cP_j$ with $i+j+1=2n-2k$. Then one has
\[
\upxb^{2n-2k} (f \uxb g ) = 2(n-k) \upxb^{2n-2k-2} \left[- (f \upxb) g + f (\upxb g) \right]. 
\]
\end{lemma}

\begin{proof}
The proof is done by induction on $k$. We first consider the case where $k = n-1$. We then need to prove that
\begin{eqnarray*}
\upxb^{2} (f_1 \uxb g_0 ) &=&  - 2 (f_1 \upxb) g_0\\
\upxb^{2} (f_0 \uxb g_1 ) &=&  2 f_0 (\upxb g_1)
\end{eqnarray*}
with $f_i, g_i \in \cP_i$. We give the proof for the first equation:
\begin{eqnarray*}
\upxb^{2} (f_1 \uxb g_0 ) &=& 4 \sum_{j=0}^{n} \partial_{{x \grave{}}_{2i-1}} \partial_{{x \grave{}}_{2i}} ( f_1 \uxb) g_0\\
&=& 4 \sum_{j=0}^{n} \partial_{{x \grave{}}_{2i-1}}  ( (\partial_{{x \grave{}}_{2i}}f_1) \uxb - f_1 {e \grave{}}_{2i}) g_0\\
&=& 4 \sum_{j=0}^{n} ( (\partial_{{x \grave{}}_{2i}}f_1) {e \grave{}}_{2i-1} - ( \partial_{{x \grave{}}_{2i-1}} f_1) {e \grave{}}_{2i}) g_0\\
&=&4 \sum_{j=0}^{n} ( (f_1 \partial_{{x \grave{}}_{2i}}) {e \grave{}}_{2i-1} - (  f_1\partial_{{x \grave{}}_{2i-1}}) {e \grave{}}_{2i}) g_0\\
&=& - 2 (f_1 \upxb) g_0.
\end{eqnarray*}
The other expression is obtained in a similar way. We proceed by induction. Suppose the lemma is proven for $k = l, \ldots, n-1$, then we prove that it also holds for $k = l-1$. This means that we have to prove that
\[
\upxb^{2n-2l+2} (f \uxb g ) = 2(n-l+1) \upxb^{2n-2l} \left[- (f \upxb) g + f (\upxb g) \right] 
\]
where $f \in \cP_i$, $g \in \cP_j$ with $i+j+1=2n-2l+2$. We do the calculation in the case where $i$ is odd (the other case is similar). Then
\begin{eqnarray*}
&&\upxb^{2n-2l+2} (f \; \uxb \; g )\\
 &=& 4 \upxb^{2n-2k} \sum_{i=1}^n \partial_{{x \grave{}}_{2i-1}} \partial_{{x \grave{}}_{2i}}(f \; \uxb \; g )\\
&=& 4 \upxb^{2n-2l} \sum_{i=1}^n \partial_{{x \grave{}}_{2i-1}} \left( (\partial_{{x \grave{}}_{2i}} f)  \uxb  g - f {e \grave{}}_{2i} g  + f \uxb (\partial_{{x \grave{}}_{2i}}g) \right )\\
&=& 4 \upxb^{2n-2l}  \sum_{i=1}^n \left( (\partial_{{x \grave{}}_{2i-1}} \partial_{{x \grave{}}_{2i}} f)  \uxb  g + (\partial_{{x \grave{}}_{2i}} f)  {e \grave{}}_{2i-1}  g - (\partial_{{x \grave{}}_{2i}} f)  \uxb  (\partial_{{x \grave{}}_{2i-1}} g) \right.\\
&& -   (\partial_{{x \grave{}}_{2i-1}} f) {e \grave{}}_{2i} g + f {e \grave{}}_{2i} (\partial_{{x \grave{}}_{2i-1}} g) + (\partial_{{x \grave{}}_{2i-1}} f) \uxb (\partial_{{x \grave{}}_{2i}}g)\\
&& \left. - f {e \grave{}}_{2i-1} (\partial_{{x \grave{}}_{2i}}g) +f \uxb (\partial_{{x \grave{}}_{2i-1}} \partial_{{x \grave{}}_{2i}}g) \right)\\
&=& 4 \upxb^{2n-2l}  \sum_{i=1}^n \left( (\partial_{{x \grave{}}_{2i-1}} \partial_{{x \grave{}}_{2i}} f)  \uxb  g + ( f \partial_{{x \grave{}}_{2i}})  {e \grave{}}_{2i-1}  g - (\partial_{{x \grave{}}_{2i}} f)  \uxb  (\partial_{{x \grave{}}_{2i-1}} g) \right.\\
&& -   ( f \partial_{{x \grave{}}_{2i-1}}) {e \grave{}}_{2i} g + f {e \grave{}}_{2i} (\partial_{{x \grave{}}_{2i-1}} g) + (\partial_{{x \grave{}}_{2i-1}} f) \uxb (\partial_{{x \grave{}}_{2i}}g)\\
&& \left. - f {e \grave{}}_{2i-1} (\partial_{{x \grave{}}_{2i}}g) +f \uxb (\partial_{{x \grave{}}_{2i-1}} \partial_{{x \grave{}}_{2i}}g) \right)\\
&=&\upxb^{2n-2l} \left[ - 2 (f \upxb) g + 2f (\upxb g) \right]\\
&&+  \upxb^{2n-2l} \left( (\upxb^2 f) \uxb g +f \uxb (\upxb^2 g) + 4 \sum_i (\partial_{{x \grave{}}_{2i-1}} f) \uxb (\partial_{{x \grave{}}_{2i}}g)- 4 \sum_i (\partial_{{x \grave{}}_{2i}} f)  \uxb  (\partial_{{x \grave{}}_{2i-1}} g)\right)
\end{eqnarray*}
where we have used the fact that for $F_i \in \cP_i$
\begin{eqnarray*}
\partial_{{x \grave{}}_{k}} F_i &=& F_i \partial_{{x \grave{}}_{k}} \quad \mbox{if $i$ is odd}\\
&=& -F_i \partial_{{x \grave{}}_{k}} \quad \mbox{if $i$ is even}.
\end{eqnarray*}
We can now apply the induction hypothesis to the last line of the previous calculation. This yields
\begin{eqnarray*}
&&\upxb^{2n-2l} \left( (\upxb^2 f) \uxb g +f \uxb (\upxb^2 g) + 4 \sum_i (\partial_{{x \grave{}}_{2i-1}} f) \uxb (\partial_{{x \grave{}}_{2i}}g)- 4 \sum_i (\partial_{{x \grave{}}_{2i}} f)  \uxb  (\partial_{{x \grave{}}_{2i-1}} g)\right)\\
&=&2(n-l)\upxb^{2n-2l-2} \left( -((\upxb^2 f)\upxb)  g +(\upxb^2 f) \upxb g   -(f \upxb) (\upxb^2 g) + f (\upxb^3 g) \right.\\
&& + 4 \sum_i [-((\partial_{{x \grave{}}_{2i-1}} f) \upxb) (\partial_{{x \grave{}}_{2i}}g) + (\partial_{{x \grave{}}_{2i-1}} f) \upxb (\partial_{{x \grave{}}_{2i}}g)]\\
&&\left. + 4 \sum_i [((\partial_{{x \grave{}}_{2i}} f)  \upxb)  (\partial_{{x \grave{}}_{2i-1}} g) - (\partial_{{x \grave{}}_{2i}} f)  \upxb (\partial_{{x \grave{}}_{2i-1}} g)]\right)\\
&=&2(n-l)\upxb^{2n-2l-2} \left( -(\upxb^2 (f\upxb))  g +(\upxb^2 f) \upxb g   -(f \upxb) (\upxb^2 g) + f (\upxb^2 \upxb g) \right.\\
&& - 4 \sum_i [(\partial_{{x \grave{}}_{2i-1}}( f \upxb)) (\partial_{{x \grave{}}_{2i}}g) + (\partial_{{x \grave{}}_{2i-1}} f)  (\partial_{{x \grave{}}_{2i}} \upxb g)]\\
&&\left. + 4 \sum_i [(\partial_{{x \grave{}}_{2i}} (f  \upxb))  (\partial_{{x \grave{}}_{2i-1}} g) + (\partial_{{x \grave{}}_{2i}} f)   (\partial_{{x \grave{}}_{2i-1}} \upxb g)] \right)\\
&=&2(n-l)\upxb^{2n-2l} \left[ - (f \upxb) g + f (\upxb g) \right].
\end{eqnarray*}
We conclude that
\begin{eqnarray*}
\upxb^{2n-2l+2} (f \; \uxb \; g ) &=& 2\upxb^{2n-2l} \left[ -  (f \upxb) g + f (\upxb g) \right]\\
&& + 2(n-l)\upxb^{2n-2l} \left[ - (f \upxb) g + f (\upxb g) \right]\\
&=& 2(n-l+1)\upxb^{2n-2l} \left[ - (f \upxb) g + f (\upxb g) \right]
\end{eqnarray*}
which completes the proof of the lemma.
\end{proof}

Using this lemma, we now calculate
\begin{eqnarray*}
&&\upxb^{2n} (f \uxb^{2k+1} g ) \\
&=& \sum_{i+j+2k+1=2n} \upxb^{2n} (f_i \uxb^{2k+1} g_{j} ) \\
&=& \sum_{i+j+2k+1=2n} c(n,k) \upxb^{2n-2k} (f_i \uxb g_j )\\
&=& \sum_{i+j+2k+1=2n} c(n,k) 2(n-k)\upxb^{2n-2k-2} \left[ - (f_i \upxb) g_j + f_i (\upxb g_j) \right]\\
&=& \sum_{i+j+2k+1=2n} 2(n-k) \frac{c(n,k)}{c(n,k+1)} \upxb^{2n} (\left[ - (f_i \upxb) g_j + f_i (\upxb g_j) \right]\uxb^{2k+2})\\
&=& 2(n-k) \frac{c(n,k)}{c(n,k+1)} \upxb^{2n} (\left[ - (f \upxb) g + f (\upxb g) \right]\uxb^{2k+2})
\end{eqnarray*}
where $f_i$ and $g_j$ are the homogeneous components of $f$ and $g$.

As 
\[
\frac{c(n,k)}{c(n,k+1)} = -\frac{1}{4(k+1) (n-k)}
\]
we conclude that
\[
-2\upxb^{2n} (f \frac{\uxb^{2k+1}}{k!} g ) = \upxb^{2n} (\left[ - (f \upxb) g + f (\upxb g) \right] \frac{\uxb^{2k+2}}{(k+1)!}).
\]

If we combine this result with the definitions of $\dsb$ and $dV(\uxb)$ we immediately obtain the fermionic Stokes' theorem:

\begin{theorem}[(Fermionic Stokes' theorem)]
Let $f$ and $g$ be elements of $\cP = \Lambda_{2n} \otimes \cC$. Then one has
\[
\int_B f \dsb g =  \int_B  [ - (f \upxb)  g + f ( \upxb  g)  ] dV(\uxb).
\]
\label{fermcauchythm}
\end{theorem}

\begin{remark}
Note that not all properties of integration in the complex plane are still valid in superspace. We do not have e.g. Morera's theorem in the purely fermionic case. Take e.g. $f = \uxb^2 P_1$ with $P_1$ a spherical monogenic of degree one. Then
\[
\int_B \dsb f = 0
\]
but $f$ is clearly not monogenic. This has to do with the fact that we have only one `contour' to be considered in purely fermionic space, whereas in the complex plane one integrates over all contours.
\end{remark}

\section{The general Stokes' theorem in Clifford analysis}
\label{generalthm}

We consider an open set $\Omega \subset \mR^m$ and a compact oriented differentiable $m$-dimensional manifold $\Sigma \subset \Omega$ with smooth boundary $\partial \Sigma$. In the previous section we have introduced two surface-elements $\ds$ and $\dsb$ and two volume-elements $dV(\ux)$ and $dV(\uxb)$. Now we want to combine these elements to obtain a suitable surface-element in superspace. First note that $\ds \dsb$ is not a good candidate, because this would yield an integration over a formal object of codimension two. It turns out that one has to define the surface element on $\Sigma$ as
\[
\dsx = \dsb \; dV(\ux) - dV(\uxb) \ds.
\]

Note that this is a vector in the $e_i$ and ${e \grave{}}_j$:
\[
\dsx = \sum_{j=1}^m (-1)^{j} e_j \widehat{dx_j} dV(\uxb) -2 \sum_{j=1}^{2n} {e \grave{}}_j {x \grave{}}_j \exp(\uxb^2)dV(\ux)
\]
which one would also expect a priori by analogy with the classical case.

The definition of $\dsx $ forces us to define integration of an object such as $f \dsx  g$ in the following way. First we introduce the notations:
\begin{eqnarray*}
\int_{B, \Sigma} &=& \int_{\Sigma} \int_B =  \int_B \int_{\Sigma}\\
\int_{B, \partial \Sigma} &=& \int_{\partial \Sigma} \int_B =  \int_B \int_{\partial \Sigma}
\end{eqnarray*}
then we define 
\begin{equation}
\int_{B, \Sigma, \partial \Sigma} f \dsx  g = \int_{B, \Sigma} f \dsb \; dV(\ux) g - \int_{B, \partial \Sigma} f dV(\uxb) \ds g 
\label{intsurfel}
\end{equation}
where we note that this includes two integrations: one over the whole manifold $\Sigma$ as well as one over the boundary $\partial \Sigma$. This might seem strange, but it is in fact very natural, as the two Berezin integrations $\int_B$ are actually one over the whole fermionic space and one over the fermionic boundary. In this way, both integrations in (\ref{intsurfel}) are formally over the odd and the even boundary of a supermanifold (in the purely bosonic case there is only an even boundary). 

Similarly we have the following general volume-element defined by
\[
d V(x) = d V(\ux) d V(\uxb).
\]
Now we can construct the general Stokes' theorem. The two terms in (\ref{intsurfel}) are calculated as follows, with $f,g \in C^1(\Omega)_{m |2n}$: 
\begin{eqnarray*}
\int_{B, \Sigma} f \dsb \; dV(\ux) g &=& \int_{B, \Sigma} (f \dsb \;  g)dV(\ux)\\
&=&\int_{\Sigma} \left[ \int_B f \dsb \;  g  \right]dV(\ux)\\
&=& \int_{\Sigma}  \left[ \int_B  [ - (f \upxb)  g + f ( \upxb  g)  ] dV(\uxb)  \right]dV(\ux)
\end{eqnarray*}
and
\begin{eqnarray*}
\int_{B, \partial \Sigma} f dV(\uxb) \ds g &=& \int_{B, \partial \Sigma} (f  \ds   g)  dV(\uxb)\\
&=& \int_B \left[ \int_{\partial \Sigma} (f  \ds   g )\right]dV(\uxb)\\
&=& \int_B \left[\int_{\Sigma} [(f \upx)g + f (\upx g) ] dV(\ux) \right] dV(\uxb).
\end{eqnarray*}

So we can calculate
\begin{eqnarray*}
\int_{B, \Sigma, \partial \Sigma} f \dsx  g &=& \int_{B, \Sigma} f \dsb \; dV(\ux) g - \int_{B, \partial \Sigma} f dV(\uxb) \ds g\\ 
&=&\int_{\Sigma}  \left[ \int_B  [ - (f \upxb)  g + f ( \upxb  g)  ] dV(\uxb)  \right]dV(\ux)\\
&& - \int_B \left[\int_{\Sigma} [(f \upx)g + f (\upx g) ] dV(\ux) \right] dV(\uxb)\\
&=& \int_{B,\Sigma}  \left[  - (f \upxb)  g  - (f \upx)g  + f ( \upxb  g) -f (\upx g)   \right] d V(x) \\
&=& \int_{B,\Sigma}  \left[  (f \px)  g + f (\px g)  \right] d V(x),
\end{eqnarray*}
where we have used the definition of the Dirac operator $\px$ acting from the left and from the right (see formulae (\ref{leftDirac}) and (\ref{rightDirac})).

Summarizing we thus obtain the following theorem.

\begin{theorem}[(General Stokes' theorem)]
Let $\Omega \subset \mR^m$ be an open set and $\Sigma \subset \Omega$ a compact oriented differentiable $m$-dimensional manifold with smooth boundary $\partial \Sigma$.
Let $f,g \in C^1(\Omega)_{m |2n}$. Then one has
\[
\int_{B, \Sigma, \partial \Sigma} f \dsx  g = \int_{B,\Sigma}  \left[  (f \px)  g + f (\px g)  \right] d V(x).
\]
\label{gencauchythm}
\end{theorem}

\section{Consequences and applications}
\label{conseq}

In this section we will discuss some corollaries of the general Stokes' theorem obtained in the previous section (see theorem \ref{gencauchythm}). First we consider the case where both $f$ and $g$ are monogenic functions. This leads to the following Cauchy theorem in superspace.

\begin{corollary}[(Cauchy theorem)]
Let $f ,g$ be right, respectively left monogenic, i.e. $f \in \cM(\Omega)_{m|2n}^{r}$, $g \in \cM(\Omega)_{m|2n}^{l}$. Then one has
\[
\int_{B, \Sigma, \partial \Sigma} f \dsx  g = 0.
\]
\end{corollary}

If we put $f$ (resp. $g$) equal to the constant function $1$, we obtain a generalization of the well-known Cauchy theorem in the complex plane, stating that for any holomorphic function $\int_{\cC} f(z) dz =0$ independently of the choice of the contour $\cC$.

\begin{corollary}
Let $f ,g$ be right, respectively left monogenic in $\Omega$. Then for every compact oriented differentiable $m$-dimensional manifold $\Sigma \subset \Omega$ with smooth boundary $\partial \Sigma$ one has
\begin{eqnarray*}
\int_{B, \Sigma, \partial \Sigma} f \dsx  &=& 0\\
\int_{B, \Sigma, \partial \Sigma} \dsx  g &=& 0.
\end{eqnarray*}
\end{corollary}

In the sequel, we will need the following lemma, the proof of which is classical.

\begin{lemma}
Let $f$ be a $C^1$-function defined in an open set $\Omega \subset \mR^m$ containing $\uy$, let $B(\uy,R)$ be a ball of radius $R$ and center $\uy$ contained in $\Omega$. Further let $\nu_k^{m|0}$ be defined as in section \ref{preliminaries}. Then the following holds:
\begin{eqnarray*}
\lim_{R\rightarrow 0+} \int_{B(\uy,R)} \nu_k^{m|0}(\ux-\uy) f(\ux) dV(\ux) &=& 0, \quad \forall k \\
\lim_{R\rightarrow 0+} \int_{\partial B(\uy,R)} \nu_k^{m|0}(\ux-\uy) \ds f(\ux)  
&=& \left\{ \begin{array}{lll} 0& \quad& \forall k > 1\\- f(\uy) & \quad&  k=1. \end{array}  \right.
\end{eqnarray*}
\label{intlemma}
\end{lemma}

Now we can formulate the following theorem.
\begin{theorem}[(Cauchy-Pompeiu)]
Let $\Omega \subset \mR^m$ be an open set and $\Sigma \subset \Omega$ a compact oriented differentiable $m$-dimensional manifold with smooth boundary $\partial \Sigma$.
Let $g \in C^1(\Omega)_{m |2n}$ and let $\nu_1^{m|2n}$ be the fundamental solution of the super Dirac operator. Then one has
\begin{eqnarray*}
&&\int_{B, \Sigma, \partial \Sigma} \nu_1^{m|2n}(x-y) \dsx g(x)  - \int_{B, \Sigma} \nu_1^{m|2n}(x-y) (\px g(x)) dV(x)\\
\vspace{-2mm}
\\&=& \left\{ \begin{array}{lll} 0 &\quad& \mbox{if  $\uy \in \Omega \backslash \Sigma$} \\
g(y) dV(\uyb)& \quad& \mbox{if $\uy \in \stackrel{\circ}{\Sigma}$}. \end{array} \right.
\end{eqnarray*}
\label{pompeiuthm}
\end{theorem}

\begin{proof}
Due to linearity it suffices to prove this formula for $g = g_1(\ux) g_2(\uxb)$ where $g_1$ contains only commuting variables and $g_2$ contains only anti-commuting variables.

The formula where $\uy \in \Omega \backslash \Sigma$ follows from a direct application of theorem \ref{gencauchythm}. So suppose $\uy \in \stackrel{\circ}{\Sigma}$. Then we consider a ball $\Gamma = B(\uy,R)$ contained in $\stackrel{\circ}{\Sigma}$ and we apply theorem \ref{gencauchythm} to $\Sigma \backslash \Gamma$. We find that 
\begin{eqnarray*}
&&\int_{B, \Sigma \backslash \Gamma, \partial (\Sigma \backslash \Gamma)} \nu_1^{m|2n}(x-y) \dsx g(x) \\
&=&\int_{B, \Sigma \backslash \Gamma} [( \nu_1^{m|2n}(x-y)\px ) g(x) +  \nu_1^{m|2n}(x-y) (\px  g(x))]dV(x)\\
&=&\int_{B, \Sigma \backslash \Gamma} \nu_1^{m|2n}(x-y) (\px  g(x)) dV(x),
\end{eqnarray*}
as $\nu_1^{m|2n}(x-y)$ is right monogenic for $\ux \neq \uy$.

If $R \rightarrow 0+$ then the right-hand side tends to 
\[
\int_{B, \Sigma } \nu_1^{m|2n}(x-y) (\px  g(x)) dV(x)
\]
because $\nu_1^{m|2n}(x-y) (\px  g(x))$ is integrable. The left-hand side is calculated as
\begin{eqnarray*}
&&\int_{B, \Sigma \backslash \Gamma, \partial (\Sigma \backslash \Gamma)} \nu_1^{m|2n}(x-y) \dsx g(x) \\
&=& \int_{B, \Sigma \backslash \Gamma} \nu_1^{m|2n}(x-y) \dsb dV(\ux) g(x) -\int_{B, \partial (\Sigma \backslash \Gamma)} \nu_1^{m|2n}(x-y) \ds dV(\uxb) g(x)\\
&=&  \int_{B, \Sigma} \nu_1^{m|2n}(x-y) \dsb dV(\ux) g(x) -\int_{B, \partial \Sigma} \nu_1^{m|2n}(x-y) \ds dV(\uxb) g(x)\\
&& - \int_{B, \Gamma} \nu_1^{m|2n}(x-y) \dsb dV(\ux) g(x) +\int_{B, \partial \Gamma} \nu_1^{m|2n}(x-y) \ds dV(\uxb) g(x)\\
&=& \int_{B, \Sigma, \partial \Sigma}  \nu_1^{m|2n}(x-y) \dsx g(x) \\
&& - \int_{B, \Gamma} \nu_1^{m|2n}(x-y) \dsb dV(\ux) g(x) +\int_{B, \partial \Gamma} \nu_1^{m|2n}(x-y) \ds dV(\uxb) g(x).
\end{eqnarray*}

Now we simplify the expression in the last line. Using lemma \ref{intlemma} we see that only the term 
\[
\nu_1^{m|0}(\ux-\uy) \frac{(\uxb-\uyb)^{2n}}{n!} = \nu_1^{m|0}(\ux-\uy) \delta(\uxb-\uyb)
\]
in $\nu_1^{m|2n}(x-y)$ will play a role. This has the following result
\begin{eqnarray*}
&& - \int_{B, \Gamma} \nu_1^{m|2n}(x-y) \dsb dV(\ux) g(x)+  \int_{B, \partial \Gamma} \nu_1^{m|2n}(x-y) \ds dV(\uxb) g(x)\\
&=&\int_{B, \partial \Gamma}  \nu_1^{m|0}(\ux-\uy) \delta(\uxb-\uyb) \ds dV(\uxb) g(x)\\
&=&\int_{\partial \Gamma} \nu_1^{m|0}(\ux-\uy) \ds g_1(\ux) \left[ \int_B  \delta(\uxb-\uyb)  dV(\uxb)  g_2(\uxb) \right]\\
&=&\int_{\partial \Gamma} \nu_1^{m|0}(\ux-\uy) \ds  g_1(\ux) dV(\uyb) g_2(\uyb)\\
&=&- g_1(\uy) dV(\uyb) g_2(\uyb)\\
&=&-g(y) dV(\uyb),
\end{eqnarray*}
when taking the limit $R \rightarrow 0 +$ and where we have applied lemma \ref{intlemma} in the penultimate line.

Putting all terms together completes the proof.
\end{proof}

\begin{remark}
The result of theorem \ref{pompeiuthm} is not completely as desired: we have found that the right-hand side equals $g(y)dV(\uyb)$. The function $dV(\uyb)$ is absent in the classical result (see e.g. \cite{MR697564}). One could propose to divide both sides of the Cauchy-Pompeiu formula by $dV(\uyb)$ to improve the result. This is however not possible, because $dV(\uyb)$ is nilpotent.
\end{remark}

If moreover $g$ is left monogenic, the Cauchy-Pompeiu theorem reduces to the following representation formula for monogenic functions in superspace.

\begin{corollary}
If $g \in \cM(\Omega)_{m|2n}^{l}$, then one has
\begin{eqnarray*}
\int_{B, \Sigma, \partial \Sigma} \nu_1^{m|2n}(x-y) \dsx g(x) &=&  \left\{ \begin{array}{lll} 0& \quad & \mbox{if  $\uy \in \Omega \backslash \Sigma$} \\
g(y) dV(\uyb) &\quad& \mbox{if $\uy \in \stackrel{\circ}{\Sigma}$}. \end{array} \right.
\end{eqnarray*}
\end{corollary}

It is easy to generalize this corollary to $k$-monogenic functions, i.e. null-solutions of $\px^k$. We obtain the following theorem, where $\nu_j^{m|2n}(x)$ denotes the fundamental solution of $\px^j$, determined in \cite{DBS6}. Note that similar formulae also exist in classical Clifford analysis, see e.g. \cite{JR3} or \cite{JR2} for the case of polynomial type Dirac operators.

\begin{theorem}
Let $\Omega \subset \mR^m$ be an open set and $\Sigma \subset \Omega$ a compact oriented differentiable $m$-dimensional manifold with smooth boundary $\partial \Sigma$.
Let $g \in C^k(\Omega)_{m |2n}$ be $k$-monogenic, i.e. $\px^k g = 0$. Then one has
\begin{eqnarray*}
\int_{B, \Sigma, \partial \Sigma} \sum_{j=1}^{k}(-1)^{j+1} \nu_j^{m|2n}(x-y) \dsx \px^{j-1} g(x) &=& \left\{ \begin{array}{ll} 0  & \mbox{if  $\uy \in \Omega \backslash \Sigma$} \\
g(y) dV(\uyb)   & \mbox{if $\uy \in \stackrel{\circ}{\Sigma}$}. \end{array} \right.
\end{eqnarray*}
\end{theorem}

\begin{proof}
Similar to the proof of theorem \ref{pompeiuthm}, using the fact that $\px  \nu_j^{m|2n} =  \nu_{j-1}^{m|2n}$.
\end{proof}

\section{Conclusions}

In this paper we have established a Cauchy integral formula in superspace. We have first obtained a Stokes' theorem in the purely fermionic case. When combining this theorem in a suitable way with the classical (bosonic) Stokes' theorem, we were able to construct the general formula. In this formula, the integration over the boundary of a supermanifold consists of two parts: one integration over the even boundary and one over the odd boundary.

Finally we used this Stokes' formula to establish a Cauchy theorem in superspace, as well as a Cauchy-Pompeiu formula, using the fundamental solution of the super Dirac operator as kernel.

\end{document}